\documentclass[reqno,notitlepage,3p]{elsarticle}

\makeatletter
\def\ps@pprintTitle{%
 \let\@oddhead\@empty
 \let\@evenhead\@empty
 \def\@oddfoot{\centerline{\thepage}}%
 \let\@evenfoot\@oddfoot}
\makeatother

\usepackage[english,french]{babel}

\usepackage{lipsum}

\newenvironment{abstracts}
 {\global\setbox\absbox=\vbox\bgroup
    \hsize=\textwidth
    \linespread{1}\selectfont}
 {\vspace{-\bigskipamount}\egroup}
\renewenvironment{abstract}[1][]
 {\if\relax\detokenize{#1}\relax\else\selectlanguage{#1}\fi
  \noindent\textbf{\abstractname}\par\medskip\noindent\ignorespaces}
 {\par\bigskip}

\usepackage{latexsym,amsfonts}
\usepackage{amsmath,amsthm}          
\usepackage{euscript}                
\usepackage{epsfig}
\usepackage[all]{xy}
\usepackage{mathrsfs}
\usepackage{calrsfs}

\usepackage[bookmarks,
bookmarksnumbered,%
 colorlinks=true,%
 linkcolor=blue,%
 citecolor=red,%
 filecolor=blue,%
 urlcolor=blue,%
]
{hyperref}

\theoremstyle{plain} \numberwithin{equation}{section}

\newtheorem{thm}{Theorem}[section]

\newtheorem{ques}[thm]{Question}

\theoremstyle{definition}

\newtheorem{rmk}[thm]{Remark}

\newcommand{\bi}{\begin{itemize}}
\newcommand{\ei}{\end{itemize}}
\newcommand{\bp}{\begin{proof}}
\newcommand{\ep}{\end{proof}}

\def\s-{\setminus}

\begin{document}

\title{On the invariance of the total Monge-Amp\`ere volume of Hermitian metrics}

\author[Chiose]{Ionu\c{t} Chiose}
\ead{Ionut.Chiose@imar.ro}
\address{
	Institute of Mathematics of the Romanian Academy,  P.O. Box 1-764, Bucharest 014700,  Romania}
	
%	\email{Ionut.Chiose@imar.ro}

\selectlanguage{english}

\begin{abstracts}
\begin{abstract}
  In this note, we describe the Hermitian metrics that leave the total Monge-Amp\`ere volume invariant.
In particular, we give several characterizations of the Hermitian metrics which satisfy the comparison principle for the complex Monge-Amp\`ere operator.
\end{abstract}
%\begin{abstract}[french]
 % Dans cette note, nous d\'ecrivons les m\'etriques hermitiennes pour lesquelles le volume total de l'op\'erateur de Monge-Amp\'ere est invariant.
  %En particulier, nous donnons plusieurs caract\'erisations des m\'etriques hermitiennes qui satisfont le principe de comparaison pour l'op\'erateur de Monge-Amp\`ere complexe.
  
%\end{abstract}
\end{abstracts}


\maketitle

Let $g$ be a Hermitian metric on a compact complex manifold $X$ of dimension $n$. Suppose that it satisfies 
the equations 
\begin{equation}\label{gl}
i\partial\bar\partial g^k=0, \forall k=\overline{1,n-1}
\end{equation}
 (we denote by the same letter $g$ the fundamental form of the Hermitian metric $g$). These metrics were introduced by Guan and Li \cite{guanli} because these metrics leave the total Monge-Amp\`ere volume unchanged, i.e., they satisfy

\begin{equation}\label{invariance}
\int_X(g+i\partial\bar\partial u)^n=\int_Xg^n,\forall u\in PSH(X,g)
\end{equation}
where $PSH(X,g)$ denotes the space of ${\mathcal C}^{\infty}$ real functions $u$ on $X$ such that $$g+i\partial\bar\partial u\geq 0$$ Indeed, the equality (\ref{invariance}) follows from (\ref{gl}) from Stokes' theorem applied twice.

In this note, we prove the converse of this statement, namely we show that a Hermitian metric which satisfies (\ref{invariance}) has to satisfy (\ref{gl}). This answers Question 29 in \cite{dinew} (see also Problem 11.1 in \cite{dinew1}).

\begin{thm} 
\label{invariant}
Let $(X,g)$ be a compact complex Hermitian manifold 
of dimension $n$ such that 
\begin{equation}
\int_X(g+i\partial\bar\partial u)^n=\int_Xg^n
\end{equation} 
for any smooth $g$-plurisubharmonic function $u$ on $X$. Then 
\begin{equation}
i\partial\bar\partial g^k=0, \forall k=\overline{1,n-1}
\end{equation} 
\end{thm}
\begin{proof} First note that if $u$ is an arbitrary ${\mathcal C}^{\infty}$ real function on $X$, then there exists $\varepsilon_{0}=\varepsilon_0(u)>0$ such that $\varepsilon u$
is $g$-plurisubharmonic for $\varepsilon\in (-\varepsilon_{0},\varepsilon_{0})$. Therefore, for any $u\in{\mathcal C}^{\infty}(X,{\mathbb R})$, we have 
\begin{equation}
\int_X(g+\varepsilon i\partial\bar\partial u)^n=\int_Xg^n,\forall \varepsilon\in(-\varepsilon_0,\varepsilon_0)
\end{equation}
If we expand the expression on the left and write it as a polynomial in $\varepsilon$, we obtain that
\begin{equation}
\int_Xg^k\wedge(i\partial\bar\partial u)^{n-k}=0,\forall k=\overline{1,n-1}
\end{equation}
Now fix $k\in\{1,2,...,n-1\}$ and if 
$u_1$, $u_2$,..., $u_{n-k}$ are arbitrary ${\mathcal C}^{\infty}$ functions on $X$ and $t_1$, $t_2$,..., $t_{n-k}$ are arbitrary real numbers, then, from 
\begin{equation}
\int_Xg^k\wedge\left (i\partial\bar\partial\sum_{i=1}^{n-k}t_i u_i\right )^{n-k}=0
\end{equation}
it follows that
\begin{equation}
\int_Xg^k\wedge i\partial\bar\partial u_1\wedge i\partial\bar\partial u_2\wedge...\wedge i\partial\bar\partial u_{n-k}=0
\end{equation}
Now use Stokes' theorem twice to obtain that, for arbitrary ${\mathcal C}^{\infty}$ functions $u_1$, $u_2$,..., $u_{n-k}$ on $X$ we have
\begin{equation}
\int_X u_1i\partial\bar\partial g^k\wedge i\partial\bar\partial u_2\wedge...\wedge i\partial\bar\partial u_{n-k}=0
\end{equation}
Since the function $u_1$ is arbitrary on $X$, we obtain that
\begin{equation}
i\partial\bar\partial g^k\wedge i\partial\bar\partial u_2\wedge...\wedge i\partial\bar\partial u_{n-k}=0
\end{equation}
on $X$ for any $u_2,..., u_{n-k}\in{\mathcal C}^{\infty}(X,{\mathbb R})$. It is clear that we can let the functions $u_2,...,u_{n-k}$ be complex valued. Fix $x\in X$ and let $z_1,z_2,...,z_n$ be local coordinates around $x$. If
\begin{equation}
i\partial\bar\partial g^k=
\sum_{\substack{i_1<...<i_{k+1}\\j_1<...j_{k+1}}}g_{i_1...i_{k+1}\bar j_1...\bar j_{k+1}}
dz_{i_1}\wedge ...\wedge dz_{i_{k+1}}\wedge d\bar z_{j_1}\wedge...\wedge d\bar z_{j_{k+1}}
\end{equation} 
then let 
\begin{equation}
\{l_1,...,l_{n-k-1}\}=\{1,...,n\}\setminus \{i_1,...,i_{k+1}\}
\end{equation}
 and 
\begin{equation}
\{m_1,...,m_{n-k-1}\}=\{1,...,n\}\setminus \{j_1,...,j_{k+1}\}
\end{equation}
 and take $u_2=z_{l_1}\bar z_{m_1}$,...,$u_{n-k}=z_{l_{n-k-1}}\bar z_{m_{n-k-1}}$ near $x$ to obtain 
\begin{equation}
g_{i_1...i_{k+1}\bar j_1...\bar j_{k+1}}(x)=0
\end{equation}
hence $i\partial\bar\partial g^k=0$.
\end{proof}

Combining the above result with \cite{dinew2}, we obtain the following

\begin{thm}\label{main}
Let $X$ be a compact complex manifold of dimension $n$ with a Hermitian metric $g$. Then the following are equivalent:
\begin{itemize}

\item[i)] $i\partial\bar\partial g=i\partial g\wedge \bar\partial g=0$

\item[ii)] $i\partial\bar\partial g=i\partial\bar\partial g^2=0$

\item[iii)] $i\partial\bar\partial g^k=0,\forall k=\overline{1,n-1}$

\item[iv)] $i\partial\bar\partial g\geq 0$, $i\partial g\wedge \bar\partial g\geq 0$

\item[v)] $\int_{\{u<v\}}(g+i\partial\bar\partial v)^n\leq\int_{\{u<v\}}(g+i\partial\bar\partial u)^n$, $\forall u,v\in PSH(X,g)$

\item[vi)] $\int_X(g+i\partial\bar\partial u)^n=\int_Xg^n$, $\forall u\in PSH (X, g)$
\end{itemize}

\end{thm}
\begin{proof}
It is clear that $i)$, $ii)$ and $iii)$ are equivalent and that they imply $iv)$. Proposition 3.2 in \cite{dinew2} shows that $iv)$ implies $v)$. 
If $u\in PSH(X,g)$, then $u-C<0<u+C$ for a constant $C$ large enough. This shows that $v)$ implies $vi)$ since $u-C, 0, u+C\in PSH(X,g)$. Finally, $vi)$ implies $iii)$ from the above Theorem \ref{invariant}.
\end{proof}

\begin{rmk}
$v)$ is called the comparison principle and it plays a crucial role in the pluripotential study of the complex Monge-Amp\`ere equation.
\end{rmk}

\begin{rmk}
The above result implies that a Hermitian metric $g$ on a compact complex manifold which satisfies  $i\partial\bar\partial g\geq 0$ and  $i\partial g\wedge \bar\partial g\geq 0$ actually satisfies $i\partial\bar\partial g=i\partial g\wedge\bar\partial g=0$, and the proof goes through Proposition 3.2 in \cite{dinew2}. However, there is a more direct proof of this fact. Indeed, from 
\begin{equation}
d(ig^{n-2}\wedge \bar\partial g)=(n-2)g^{n-3}\wedge i\partial g\wedge\bar\partial g+g^{n-2}\wedge i\partial\bar\partial g
\end{equation}
and Stokes' theorem, it follows that 
\begin{equation}
(n-2)\int_Xg^{n-3}\wedge i\partial g\wedge\bar\partial g+\int_Xg^{n-2}\wedge i\partial\bar\partial g=0
\end{equation}
and, since both integrals are positive, it follows that they have to be zero. Therefore $g^{n-3}\wedge i\partial g\wedge\bar\partial g=g^{n-2}\wedge i\partial\bar\partial g=0$, that is the trace measures of the positive currents $i\partial g\wedge\bar\partial g$ and $i\partial\bar\partial g$ with respect to $g$ are zero, hence they have to be zero.
\end{rmk}

\begin{rmk}
If a compact complex manifold admits a Hermitian metric $g$ as above, i.e., it satisfies $i\partial\bar\partial g=i\partial\bar\partial g^2=0$, and if $p:Y\to X$ is a blowup with a smooth center, then $Y$ supports a Hermitian metric $g'$ with the same property. Indeed, if $E$ is the exceptional divisor, then the line bundle ${\mathcal O}(-[E])$ has a metric with curvature $\beta$ such that $g'=Np^*g+\beta$ is positive for $N$ a large enough constant. Then it is trivial to show that $g'$ satisfies the same equations as $g$. However, it is not true that this property is a bimeromorphic invariant, and this follows from the fact that a manifold bimeromorphic to a K\"ahler manifold need not be K\"ahler \cite{hironaka}, combined with the fact that a Fujiki manifold which supports a $SKT$ metric is K\"ahler \cite{chiose}.
\end{rmk}
\begin{rmk}
It is fairly easy to construct metrics as in Theorem \ref{main}. Take $(X,g)$ to be a manifold as in Theorem \ref{main} and $(Y,h)$ a K\"ahler manifold. If $p_X$ and $p_Y$ denote the two projections, then the metric $p_X^*g+p^*_Yh$ also satisfies the above equations on $X\times Y$. Now take $g$ to be a Gauduchon metric on a non-K\"ahler surface $X$ to obtain examples of non-K\"ahler manifolds that admit metrics as above.
\end{rmk}
\begin{rmk}
A similar question is the following:

\begin{ques}
Characterize the Hermitian metrics $g$ on a compact complex manifold $X$ such that there exists a constant $C$ such that $$\int_X(g+i\partial\bar\partial u)^n\leq C, \forall u\in PSH(X,g)$$
\end{ques}\noindent
This question is related to the proof of Theorem 4.1 in \cite{chiose1}. For instance, on surfaces, any Hermitian metric has the above property, while on $3$-folds, this class includes the Hermitian metrics $g$ that satisfy $i\partial\bar\partial g\geq 0$. Indeed, if $u\in PSH(X,g)$, then $$\int_X(g+i\partial\bar\partial u)^3=\int_Xg^3+3\int_Xg^2\wedge i\partial\bar\partial u+3\int_Xg\wedge  ( i\partial\bar\partial u )^2$$
and the second integral on the right is known to be bounded, while the third integral is negative if $i\partial\bar\partial g\geq 0$ since 
$$\int_X g\wedge (i\partial\bar\partial u)^2=-\int_Xi\partial\bar\partial g\wedge i\partial u\wedge\bar\partial u \leq 0$$
\noindent
from Stokes' theorem. In particular, as in \cite{chiose1}, we obtain that if a Hermitian $3$-fold $(X,g)$ such that $i\partial\bar\partial g\geq 0$ admits a nef class $\alpha$ of non-negative self-intersection, then $X$ is K\"ahler (see Theorem 2.3 in  \cite{chiose}).
\end{rmk}

\section*{Acknowledgments}
The author was supported by 
 a CNCS - UEFISCDI grant, project no.
PN-III-P4-ID-PCE-2016-0341.

%%%%%%%%%%%%%%%%%%%%%%%%%%%%%%%%
%\subsection*{Acknowledgements}  The author was supported by 
% a CNCS - UEFISCDI grant, project no.
%PN-III-P4-ID-PCE-2016-0341.

\end{document}